\documentclass[a4paper,12pt]{article}

\usepackage{amsmath,amssymb,amsthm, mathrsfs}
\usepackage{latexsym}
\usepackage{graphics}
\usepackage{hyperref}
\usepackage{authblk} 
\usepackage[numbers]{natbib}

\theoremstyle{plain}
\newtheorem{Thm}{Theorem}[section]
\newtheorem{Lem}[Thm]{Lemma}
\newtheorem{Prop}[Thm]{Proposition}
\newtheorem{Cor}[Thm]{Corollary}

\theoremstyle{definition}
\newtheorem{Def}[Thm]{Definition}

\usepackage{standalone}
\usepackage{pgfpages,tikz,pgfkeys,pgfplots}
\usetikzlibrary{arrows,positioning,matrix,fit,backgrounds,shapes,shapes.geometric, intersections}

\usetikzlibrary{shapes.callouts,decorations.pathmorphing} 
\usetikzlibrary{shadows} 
\usepackage{calc} 
\usetikzlibrary{decorations.markings} 
\tikzstyle{vertex}=[circle, draw, inner sep=0pt, minimum size=6pt] 
\newcommand{\vertex}{\node[vertex]}
\usetikzlibrary{arrows,matrix} 

\tikzset{node distance=2cm, auto}
\usetikzlibrary{positioning,fit,backgrounds,shapes,chains,calc}

\newcommand{\RRR}{\mathcal{R}} 
\newcommand{\AAA}{\mathcal{A}} 
\newcommand{\BBB}{\mathcal{B}} 

\title{The niche graphs of bipartite tournaments}

\author{\textsc{Soogang Eoh}%
\footnote{Department of Mathematics Education,
Seoul National University, Seoul 08826, Korea.
\textit{E-mail}: \texttt{mathfish@snu.ac.kr}}
\quad
\textsc{Jihoon Choi}%
\footnote{Department of Mathematics Education,
Cheongju University, Cheongju 28503, Korea.
\textit{E-mail}: \texttt{jihoon@cju.ac.kr}}
\quad
\textsc{Suh-Ryung Kim}%
\footnote{Department of Mathematics Education,
Seoul National University, Seoul 08826, Korea.
\textit{E-mail}: \texttt{srkim@snu.ac.kr}}
\quad
\textsc{Miok Oh}%
\footnote{Department of Mathematics Education,
Seoul National University, Seoul 08826, Korea.
\textit{E-mail}: \texttt{amieoki0@snu.ac.kr}}
}
\begin{document}
\maketitle
\begin{abstract}
In this paper, we completely characterize the niche graphs of bipartite tournaments and find their interesting properties.
\end{abstract}

\section{Introduction}
In this paper, a graph means a simple graph.

The  {\em niche graph} of a digraph $D$ is a graph with vertex set $V(D)$ and  edge set  $\{uv \mid (u,w) \in A(D) \text{ and } (v,w) \in A(D), \text{ or } (w,u) \in A(D) \text{ and } (w,v) \in A(D) \text{ for some } w \in V(D)\}$. The notion of niche graph is a variant of competition graph. The \emph{competition graph} of $D$
is the graph having vertex set $V(D)$ and edge set $\{ uv \mid (u,w) \in A(D) \text{ and } (v,w) \in A(D) \text{ for some } w \in V(D)\}$.
Cohen~\cite{cohen} introduced the notion of competition graph while studying predator-prey concepts in ecological food webs. Cohen's empirical observation that real-world competition graphs are usually interval graphs had led to a great deal of research on the structure of competition graphs and on the relation between the structure of digraphs and their corresponding competition graphs. In the same vein, various variants of competition graph have been introduced and studied, one of which is the notion of niche graph  introduced by Cable~{\it et al.}~\cite{Cable} (see \cite{m-step, tournament, p-competition, park2014niche, RobertsSheng, Scott} for other variants of competition graph). For work on this topic, see~\cite{bowser1991some, bowser1999niche, fishburn1992niche, seager1998cyclic}.

An orientation of a complete bipartite graph is sometimes called a \emph{bipartite tournament} and we use whichever of the two terms is more suitable for a given situation throughout this paper.

Kim~{\it et al.}~\cite{Kim} and Choi~{\it et al.}~\cite{12stepbipartite} studied the competition graphs of bipartite tournaments and the $(1,2)$-step competition graphs of bipartite tournaments, respectively.
In this paper, we study the niche graphs of bipartite tournaments to extend the work done  by Bowser~{\it et al.}~\cite{bowser1999niche} who studied niche graphs of tournaments.

If a graph is the niche graph of a bipartite tournament, then we say that it is \emph{niche-realizable through a bipartite tournament}
(in this paper, we only consider bipartite tournaments
and so we omit ``through a bipartite tournament").

In Section~2, we introduce two relations on the vertex set of a bipartite tournament to utilize in characterizing the niche graph of a bipartite tournament.
Then we present fundamental properties of niche graphs of bipartite tournaments which are immediately obtained by properties of those relations.
In Section~3, we give a complete characterization of niche-realizable graphs. First of all, we show that a graph having three or four components is niche-realizable if and only if each of its components is complete.
Then we give our main result which completely characterizes niche-realizable graphs with exactly two components. In Section~4, we find meaningful properties of niche-realizable graphs based on the characterization of niche-realizable graphs obtained in Section~3.

\section{Relations on the vertex set of a bipartite tournament arising from its niche graph}

In this section, we introduce two relations on the vertex set of a bipartite tournament to utilize in characterizing the niche graph of a bipartite tournament.
Then we present fundamental properties of niche graphs of bipartite tournaments which are immediately obtained by properties of those relations.

\begin{Prop} \label{prop:noedge}
Let $D$ be a bipartite tournament with bipartition $(U, V)$.
Then the niche graph of $D$ has no edges between the vertices in $U$ and the vertices in $V$.
\end{Prop}

\begin{proof}
Take a vertex $u$ in $U$
and a vertex $v$ in $V$.
Then $N^+_D(u) \cup N^-_D(u) = V$ and $N^+_D(v) \cup N^-_D(v) = U$ and therefore $N^+_D(u) \cap N^+_D(v) \subseteq V \cap U$ and $N^-_D(u) \cap N^-_D(v) \subseteq V \cap U$.
Since $U \cap V = \emptyset$, $N^+_D(u) \cap N^+_D(v)=\emptyset$ and $N^-_D(u) \cap N^-_D(v)=\emptyset$.
Thus $u$ and $v$ are not adjacent in the niche graph of $D$.
\end{proof}
A {\it union} $G \cup H$ of two graphs $G$ and $H$ is the graph having its vertex set $V(G) \cup V(H)$ and edge set $E(G) \cup E(H)$.
In this paper, we mean by $G \cup H$ the disjoint union of $G$ and $H$.

Suppose that a graph $G$ is niche-realizable. Then $G$ is the niche graph of a bipartite tournament.
Let $(U,V)$ be its bipartition.
By Proposition~\ref{prop:noedge}, $G$ is a disjoint union of $G[U]$ and $G[V]$ and we have the following corollary.
\begin{Cor}\label{cor:at least two components}
If a graph is  niche-realizable, then it is a disjoint union of two graphs.
\end{Cor}

Based on Proposition~\ref{prop:noedge}, we introduce the following definition.
\begin{Def}
Let $G_1$ and $G_2$ be two graphs with $m$ vertices and $n$ vertices, respectively.
The pair $(G_1,G_2)$ is said to be \emph{niche-realizable through $K_{m,n}$}
(in this paper, we only consider orientations of $K_{m,n}$
and so we omit ``through $K_{m,n}$")
if the disjoint union of $G_1$ and $G_2$ is
the niche graph of an orientation of the complete bipartite graph $K_{m,n}$
with bipartition $(V(G_1),V(G_2))$.
\end{Def}



Let $D$ be a bipartite tournament with bipartition $(U, V)$.
Then
\begin{equation}\label{eqn:equ1}
N_D^+(u)=N_D^+(v) \Leftrightarrow N_D^-(u)=N_D^-(v)
\end{equation}
and
\begin{equation}\label{eqn:equ2}
N_D^+(u)=N_D^-(v) \Leftrightarrow N_D^-(u)=N_D^+(v)
\end{equation}
for vertices $u$ and $v$ in the same partite set of $D$.
Moreover,
\begin{itemize}
\item[($\natural$)] $N^+_D(u) = N^+_D(v)$ implies that $u$ and $v$ belong to the same partite set of $D$.
\end{itemize}
We define a relation $\equiv_D$ on $V(D)$ by
\begin{equation}\label{eqn:equivalence}
u \equiv_D v \Leftrightarrow N^+_D(u)=N^+_D(v)
\end{equation}
for $u$ and $v$ in $V(D)$.
Clearly, $\equiv_D$ is an equivalence relation on $V(D)$.
For a vertex $u$ in $V(D)$, we denote the equivalence class containing $u$ by $[u]_D$.
It is obvious that
$
U \subset \bigcup_{u \in U}[u]_D \quad \text{and} \quad V \subset \bigcup_{v \in V}[v]_D.
$
Now we take a vertex $x$ in $\bigcup_{u \in U}[u]_D$.
Then $x \in [u]_D$ for some $u \in U$.
By definition, $N^+_D(u) = N^+_D(x)$.
By ($\natural$), $x \in U$, so $\bigcup_{u \in U}[u]_D \subset U$.
Therefore $U = \bigcup_{u \in U}[u]_D$.
By a similar argument, we also have
$V = \bigcup_{v \in V}[v]_D$.
Thus $$
U = \bigcup_{u \in U}[u]_D \quad \text{and} \quad V = \bigcup_{v \in V}[v]_D.
$$

Let $G$ be a graph.
Two vertices $u$ and $v$ of $G$ are said to be \emph{homogeneous} if they have the same closed neighborhood, and denoted by $u \equiv_G v$.
A clique $K$ of $G$ is said to be \emph{homogeneous} if the vertices in $K$ are mutually homogeneous.
In this paper, we mean by a clique of a graph $G$ a complete subgraph of $G$ or its vertex set.
A maximal homogeneous clique is said to be \emph{critical}.
It is easy to check that the relation $\equiv_G$ on $V(G)$ is an equivalence relation, and that every critical clique in $G$ is in fact equal to an equivalence class under $\equiv_G$.
Therefore any two critical cliques are vertex-disjoint.

Then, it is easy to check that, for the niche graph $G$ of $D$,
\begin{itemize}
\item[($\star$)] if $u \equiv_D v$, then $u$ and $v$ are adjacent and homogeneous in $G$.
\end{itemize}
We also introduce another relation $\RRR_D$ on $V(D)$ defined by
\[
u \, \RRR_D \, v \mbox{ if and only if $u$ and $v$ belong to the same partite set of $D$ and } N^+_D(u)=N^-_D(v)
\]
for $u$ and  $v$ in $V(D)$.
By \eqref{eqn:equ2}, $\RRR_D$ is symmetric.
For a vertex $u$ in $V(D)$, we let
\[
\RRR_D(u) = \{v \in V(D) \mid u \, \RRR_D \, v\}.
\]



\begin{Prop}\label{prop:relation}
Let 
$D$ be a bipartite tournament.
Then, for $u$ and $v$ in $V(D)$, the following are equivalent:
\begin{itemize}
  \item[(i)]  $u \, \RRR_D \, v$
  \item[(ii)] $\RRR_D(u)=[v]_D$
  \item[(iii)] $\RRR_D(v)=[u]_D$
\end{itemize}
\end{Prop}
\begin{proof}
Let $(U,V)$ be the bipartition of $D$.
We will show that (i) $\Leftrightarrow$ (ii).
To show the ``if'' part, take $u$ and $v$ in $V(D)$.
Suppose $\RRR_D(u)=[v]_D$. Since $v \in [v]_D$, $v \in \RRR_D(u)$ and so $u \, \RRR_D \, v$.
Thus the ``if'' part is true.

To show the ``only if'' part, suppose that $u \, \RRR_D \, v$ for some $u$ and $v$ in $V(D)$.
Then $u$ and $v$ belong to the same partite set of $D$ and $N^+_D(u)=N^-_D(v)$.
Without loss of generality, we may assume $u$ and $v$ belong to $U$.
By \eqref{eqn:equ2}, we have $N^-_D(u)=N^+_D(v)$.
Now we take $x \in V(D)$.
Then
\begin{align*}
x \in \RRR_D(u) &\Leftrightarrow x \in U \text{ and }  \ N^+_D(u)=N^-_D(x) \tag{the definition of $\RRR_D(u)$} \\
&\Leftrightarrow  x \in U \text{ and } \ N^-_D(u)=N^+_D(x) \tag{\eqref{eqn:equ2}} \\
&\Leftrightarrow N^+_D(v)=N^+_D(x) \tag{($\natural$) \text{and} $N^-_D(u)=N^+_D(v)$} \\
&\Leftrightarrow x \in [v]_D \tag{the definition of $\equiv_D$}
\end{align*}
and the ``only if'' part is true.


By the fact that (i) $\Leftrightarrow$ (ii), $v \, \RRR_D \, u \Leftrightarrow  \RRR_D(v) = [u]_D$.
Then, since $\RRR_D$ is symmetric,
(ii) $\Leftrightarrow$ (iii) is true.
\end{proof}

\begin{Cor}\label{cor:sameclass}
Let $D$ be a bipartite tournament.
If $u \, \RRR_D \, v$ and $v \, \RRR_D \, w$ for distinct $u,v,w \in V(D)$,
then $[u]_D=[w]_D$ and $u$ and $w$ are adjacent in the niche graph of $D$.
\end{Cor}
\begin{proof}
Suppose that $u \, \RRR_D \, v$ and $v \, \RRR_D \, w$ for some $u,v,w \in V(D)$.
Then, by Proposition~\ref{prop:relation}, $\RRR_D(v) = [u]_D$ and $\RRR_D(v) = [w]_D$.
Thus $[u]_D=[w]_D$.
Hence, by ($\star$), $u$ and $w$ are adjacent in the niche graph of $D$.
\end{proof}


\begin{Prop}\label{prop:notedge}
Let $G$ be the niche graph of a bipartite tournament $D$.
Then, for two vertices $u$ and $v$ in the same partite set of $D$, $uv \notin E(G)$ if and only if $u \, \RRR_D \, v$.
\end{Prop}
\begin{proof}
Take $u$ and $v$ in the same partite set of $D$.
By the definition of niche graph, $uv \notin E(G)$ if and only if $N^+_D(u) \cap N^+_D(v) = \emptyset$ and $N^-_D(u) \cap N^-_D(v) = \emptyset$.
Since $D$ is a bipartite tournament, the right hand side of the above equivalence is equivalent to $N^+_D(u) = N^-_D(v)$, which is equivalent to $u \, \RRR_D \, v$.
\end{proof}

\begin{Thm}\label{thm:alpha}
Let $(G_1,G_2)$ be a niche-realizable pair.
Then $\alpha(G_1) \le 2$ and $\alpha(G_2) \le 2$.
\end{Thm}
\begin{proof}
Since $(G_1,G_2)$ is a niche-realizable pair, $G_1 \cup G_2$ is the niche graph of a bipartite tournament $D$.
By symmetry, it suffices to show that $\alpha(G_1) \le 2$.
If $\{u, v, w\}$ is an independent set of $G_1$,
then $u \, \RRR_D \, v$ and $v \, \RRR_D \, w$ by Proposition~\ref{prop:notedge} and so, by Corollary~\ref{cor:sameclass}, $u$ and $w$ are adjacent in $G_1$, which is a contradiction.
Hence $\alpha(G_1) \le 2$.
\end{proof}

The following four corollaries immediately follow from Theorem~\ref{thm:alpha}.

\begin{Cor}\label{cor:components4}
Let $(G_1,G_2)$ be a niche-realizable pair.
Then each of $G_1$ and $G_2$ has at most two components.
\end{Cor}

\begin{Cor}\label{cor:alpha2}
Let $G$ be a niche-realizable graph and $H$ be a component of $G$.
Then $\alpha(H) \le 2$.
\end{Cor}

\begin{Cor}
A niche-realizable graph is $K_{1,3}$-free.
\end{Cor}

An independent set of three vertices such that each pair is joined by a path that avoids the neighborhood of the third is called an \emph{asteroidal triple}.

\begin{Cor}\label{cor:AT-free}
A niche-realizable graph does not contain an asteroidal triple as an induced subgraph.
\end{Cor}

A \emph{hole} of a graph is an induced cycle of length greater than or equal to four.

\begin{Cor}
A niche-realizable graph is chordal if and only if it is interval.
\end{Cor}
\begin{proof}
The ``if'' part is obvious.
To show the ``only if'' part, suppose that
a niche-realizable graph $G$ is chordal.
Then $G$ does not contain a hole.
By Corollary~\ref{cor:AT-free}, $G$ does not contain an asteroidal triple as an induced subgraph.
By the characterization of an interval graph given in \cite{LB}, $G$ is interval.
\end{proof}

\begin{Prop}\label{prop:induced P3}
If a graph is niche-realizable, then it has no induced path of length three.
\end{Prop}
\begin{proof}
Let $G$ be a niche-realizable graph. Then $G$ is the niche graph of a bipartite tournament, say $D$.
Let $(U,V)$ be the bipartition of $D$. Let $G_1$ and $G_2$ be the subgraphs of $G$ induced by $U$ and $V$, respectively.
Suppose, to the contrary, that $G$ contains an induced path $P:=xyzw$.
Without loss of generality, we may assume $P$ is a path of $G_1$.
Since $P$ is an induced path, by Proposition~\ref{prop:notedge}, $x \, \RRR_D \, z$ and $x \, \RRR_D \, w$.
By Corollary~\ref{cor:sameclass}, $[z]_D=[w]_D$.
Thus, by the fact that $y$ and $z$ are adjacent on $P$, $w$ and $y$ are adjacent in $G$, which contradicts the assumption that $P$ is an induced path.
\end{proof}

The following corollary is an immediate consequence of Proposition~\ref{prop:induced P3}.

\begin{Cor}\label{diameter}
Every component of a niche-realizable has diameter at most two.
\end{Cor}

\begin{Prop} \label{prop:complement forbidden path}
If a graph $G$ is niche-realizable, then each of the components of $\overline{G}$ has diameter at most two.
\end{Prop}
\begin{proof}
Let $G$ be a niche-realizable graph.
Then $G$ is the niche graph of a bipartite tournament $D$.
Take two vertices $x$ and $y$ in $G$.
If $x$ and $y$ belong to distinct partite sets od $D$, then they are not adjacent in $G$ by Proposition~\ref{prop:noedge} and so they are adjacent in $\overline{G}$.
Therefore, if $x$ and $y$ belong to distinct partite sets of $D$, then $d_{\overline{G}}(x,y)=1$.
Suppose that $x$ and $y$ belong to the same partite set of $D$.
Then we take a vertex $z$ in the other partite set of $D$.
Then $z$ is adjacent to both $x$ and $y$ in $\overline{G}$ by the previous observation, so $d_{\overline{G}}(x,y) \le 2$.
\end{proof}

\section{Characterizations of niche-realizable graphs}
In this section, we completely characterize niche-realizable graphs.
By Corollaries~\ref{cor:at least two components} and~\ref{cor:components4}, every niche-realizable graph has at least two and at most four components.
Based on this observation, we first characterize a niche-realizable graph with three or four components by showing that it is niche-realizable if and only if each of its components is complete.
Then we completely characterize a niche-realizable graph with exactly two components, which is the main result of this section.

Given a digraph $D$ and vertex sets $S$ and $T$ of $D$, we denote the set of arcs from $S$ to $T$ by $[S, T]$, that is, $[S, T] =\{(x,y) \in A(D)  \mid x \in S \text{ and } y \in  T\}$.

\begin{Lem}\label{lem:K_iKm-iKn}
For positive integers $i$, $j$, and $k$,
the pair $(K_i \cup K_{j}, K_k)$ is niche-realizable. 
\end{Lem}
\begin{proof}
Fix positive integers $i$, $j$, and $k$. 
We define a bipartite tournament $D$ with bipartition $(U, V)$ satisfying $|U|=i+j$, $|V|=k$ in the following.
Let $S$ be a subset of $U$ of size $i$ and $[S, V] \cup [V, U\setminus S]$ be the arc set of $D$.
Then the niche graph of $D$ is $(K_i \cup K_{j}) \cup K_k$. Thus the pair $(K_i \cup K_{j}, K_k)$ is niche-realizable.
\end{proof}

\begin{Lem}\label{lem:K_iK_m-iK_jk_m-j}
For positive integers $i$, $j$, $k$, and $l$, 
the pair $(K_i \cup K_{j}, K_k \cup K_l)$ is niche-realizable. 
\end{Lem}
\begin{proof}
Fix positive integers $i$, $j$, $k$, and $l$.
We define a bipartite tournament $D$ with bipartition $(U,V)$, where $|U|=i+j$ and $|V|=k+l$ as follows.

Let $S$ and $T$ be subsets of $U$ and $V$, respectively, where $|S|=i$, $|T|=k$ and let
\[A(D)=[S,T] \cup [T,U\setminus S] \cup [U\setminus S, V\setminus T] \cup [V\setminus T, S].\]
Then it is easy to check that the niche graph of $D$ is $(K_i \cup K_{j}) \cup (K_k \cup K_{l})$.
Thus the pair $(K_i \cup K_{j}, K_k \cup K_{l})$ is niche-realizable.
\end{proof}

\begin{Lem}\label{lem:completecomponent}
Suppose that the disjoint union of two graphs $G_1$ and $G_2$ is the niche graph of a bipartite tournament $D$ with bipartition $(U, V)$ such that $G_1 = G[U]$ and $G_2 = G[V]$.
If $G_1$ has exactly two components, then those two components are complete graphs with vertex sets $[x]_D$ and $[y]_D$, respectively, for some distinct vertices $x$ and $y$ in $G_1$.
\end{Lem}
\begin{proof}
Suppose $G_1$ has two components $X_1$ and $X_2$.
By Theorem~\ref{thm:alpha}, $X_1$ and $X_2$ are complete.
Let $x$ and $y$ be vertices in $X_1$ and $X_2$, respectively.
Since $x$ and $y$ are nonadjacent in $G_1$, $x \, \RRR_D \, y$ by Proposition~\ref{prop:notedge}.
Now take a vertex $z$ in $X_1$.
Then $z \, \RRR_D \, y$ by Proposition~\ref{prop:notedge}.
Thus $x \, \RRR_D \, y$ and $z \, \RRR_D \, y$ and so, by Corollary~\ref{cor:sameclass}, $[x]_D = [z]_D$.
By symmetry, for any vertex $w$ in $X_2$, $[y]_D = [w]_D$.
Therefore $V(X_1) \subset [x]_D$ and $V(X_2) \subset [y]_D$.
Thus
\[
U = V(X_1) \dot \cup V(X_2) \subset [x]_D \dot \cup [y]_D \subset U,
\]
where $\dot \cup$ means a disjoint union of two sets,
and so $V(X_1) = [x]_D$ and $V(X_2) = [y]_D$.
\end{proof}

\begin{Thm}\label{thm:component:new}
Let $G$ be a graph having three or four components.
Then $G$ is niche-realizable if and only if each of its components is complete.
\end{Thm}

\begin{proof}
The ``if" part is true by Lemmas~\ref{lem:K_iKm-iKn} and \ref{lem:K_iK_m-iK_jk_m-j}.
To show the ``only if'' part, suppose that $G$ is niche-realizable.
Then $G$ is the niche graph of a bipartite tournament, say $D$.
Let $(U,V)$ be the bipartition of $D$.
By the hypothesis, $G[U]$ or $G[V]$ has at least two components.
Without loss of generality, we may assume that $G[U]$ has at least two components.
Then, by Corollary~\ref{cor:components4}, $G[U]$ has exactly two components.
By Lemma~\ref{lem:completecomponent}, they are complete graphs and
\begin{equation}\label{eqn:Uxy}
U = [x]_D \dot \cup [y]_D
\end{equation}
for some vertices $x$ and $y$ in $U$ satisfying $x \, \RRR_D \, y$.
Since $D$ is a bipartite tournament, $V=N^+_D(x) \dot \cup N^-_D(x)$.
If one of $N^+_D(x)$ and $N^-_D(x)$ is empty, then $V$ equals $N^+_D(x)$ or $N^-_D(x)$ and so $G[V]$ is complete, which implies that the ``only if'' part is true.
Suppose that both $N^+_D(x)$ and $N^-_D(x)$ are nonempty.
We take
\begin{equation}\label{eqn:component1}
u \in N^+_D(x) \quad \mbox{and} \quad v \in N^-_D(x).
\end{equation}
By~\eqref{eqn:component1} and the definition of $\equiv_D$, $u \in N^+_D(z)$ and $v \in N^-_D(z)$, or equivalently $z \in N^-_D(u)$ and $z \in N^+_D(v)$ for any $z \in [x]_D$.
Thus $[x]_D \subset N^-_D(u)$ and $[x]_D \subset N^+_D(v)$.
Since $x \, \RRR_D \, y$, \eqref{eqn:component1} implies $u \in N^-_D(y)$ and $v \in N^+_D(y)$ and, by the same argument as above, we have $[y]_D \subset N^+_D(u)$ and $[y]_D \subset N^-_D(v)$.
Thus
\[
[x]_D \dot \cup [y]_D \subset N^-_D(u) \dot \cup N^+_D(u) = U \quad \text{and} \quad [x]_D \dot \cup [y]_D \subset N^+_D(v) \dot \cup N^-_D(v) = U.
\]
By \eqref{eqn:Uxy}, $[x]_D=N^-_D(u)=N^+_D(v)$ and $[y]_D=N^+_D(u)=N^-_D(v)$, which imply $u \, \RRR_D \, v$.
Hence, by Proposition~\ref{prop:notedge}, $u$ and $v$ are nonadjacent in $G$.
Since $u$ and $v$ are arbitrarily chosen, $N^+_D(x)$ and $N^-_D(x)$ are complete components of $G[V]$ and the the ``only if'' part follows.
\end{proof}

\begin{Lem} \label{lem:subgraph}
Let $G$ be the niche graph of a bipartite tournament $D$ with bipartition $(U,V)$ and $Z$ be a set of representatives of all the equivalence classes under $\equiv_D$ given in \eqref{eqn:equivalence}. Then each component of the subgraph  of $G$ induced by $Z$ is a complete multipartite graph with the partite sets of size at most two.
\end{Lem}
\begin{proof}
Let $H$ be the subgraph  of $G$ induced by $Z$ and $X$ be a component of $H$.
By Proposition~\ref{prop:noedge},
$V(X) \subset U$ or $V(X) \subset V$.
If $X$ is complete, it is a complete multipartite graph with each partite set of size one.
Now suppose that $X$ is not complete.
Then there are nonadjacent vertices in $X$.
Let $x_1$ and $x_2$ be nonadjacent vertices in $X$.
Then, since $V(X) \subset U$ or $V(X) \subset V$,  $x_1 \, \RRR_D \, x_2$ by Proposition~\ref{prop:notedge}.
By symmetry, $x_2 \, \RRR_D \, x_1$.
Since $X$ is connected and $x_1$ and $x_2$ are not adjacent in $X$, there exist vertices in $X$ distinct from $x_1$ and $x_2$.
Now take one of them and denote it by $x$.
Suppose $x$ is not adjacent to $x_1$ in $X$.
Then $x \, \RRR_D \, x_1$ by Proposition~\ref{prop:notedge}. 
Since $x_1 \, \RRR_D \, x_2$ and $x \, \RRR_D \, x_1$, by Corollary~\ref{cor:sameclass}, $[x_2]_D = [x]_D$, which is impossible as $x$ and $x_2$ were chosen as representatives of distinct equivalence classes.
Thus $x$ is adjacent to $x_1$ in $X$.
By symmetry, $x$ is adjacent to $x_2$ in $X$.
Suppose that there is a set of two nonadjacent vertices in $X$ distinct from $\{x_1,x_2\}$.
Denote those two nonadjacent by $x_3$ and $x_4$.
Then $\{x_1,x_2\} \cap \{x_3, x_4\} = \emptyset$ since $x_1$ and $x_2$ are adjacent to each vertex in $X \setminus \{x_1, x_2\}$.
By the same argument as above, we may show that $x_3$ and $x_4$ are adjacent to every vertex in $X \setminus \{x_3, x_4\}$.
We may repeat this process of selecting a set of two nonadjacent vertices in $X$ until we cover every pair of nonadjacent vertices in $X$ and the lemma statement is true.
\end{proof}

For a graph $G$, a vertex $v$ of $G$, and a finite set $K$ disjoint from $V(G)$, we say that $v$ \emph{is replaced with a clique} formed by $K$ to obtain a new graph  with the vertex set $(V(G) \cup K) \setminus \{v\}$ and the edge set $E(G-v) \cup \{wx \mid w \neq x, \{w,x\}\subset K\} \cup \{uw \mid uv \in E(G), w\in K\}$.
See Figure~\ref{fig:replace} for an illustration.
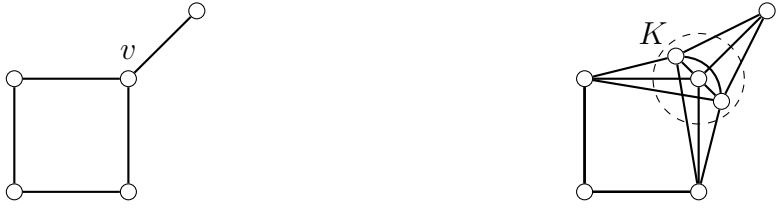
\begin{figure}
\begin{center}

\begin{tikzpicture}[x=1.5cm, y=1.5cm]

    \vertex (x1) at (0,0) [label=above:$$]{};
    \vertex (x2) at (0,1) [label=above:$$]{};
    \vertex (x3) at (1,0) [label=above:$$]{};
    \vertex (v) at (1,1) [label=above:$v$]{};
    \vertex (x5) at (1.6,1.6) [label=above:$$]{};

    \path
    (x1) edge [-,thick] (x2)
    (x1) edge [-,thick] (x3)
    (v) edge [-,thick] (x3)
    (v) edge [-,thick] (x2)
    (v) edge [-,thick] (x5)
	;

    \vertex (y1) at (5,0) [label=above:$$]{};
    \vertex (y2) at (5,1) [label=above:$$]{};
    \vertex (y3) at (6,0) [label=above:$$]{};
    \vertex (v1) at (6,1) [label=above:$$]{};
    \vertex (v2) at (6.2,0.8) [label=above:$$]{};
    \vertex (v3) at (5.8,1.2) [label=above:$$]{};
    \vertex (y5) at (6.6,1.6) [label=above:$$]{};

    \draw[dashed] (6,1) circle (0.4) [label=right:$K$]{};
    \draw (5.6,1.4) node{$K$};

    \path

    (v1) edge [-,thick] (v2)
    (v2) edge [-,bend right=40,thick] (v3)
    (v3) edge [-,thick] (v1)

    (y1) edge [-,thick] (y2)
    (y1) edge [-,thick] (y3)
    (v1) edge [-,thick] (y3)
    (v1) edge [-,thick] (y2)
    (v1) edge [-,thick] (y5)

    (y1) edge [-,thick] (y2)
    (y1) edge [-,thick] (y3)
    (v2) edge [-,thick] (y3)
    (v2) edge [-,thick] (y2)
    (v2) edge [-,thick] (y5)

    (y1) edge [-,thick] (y2)
    (y1) edge [-,thick] (y3)
    (v3) edge [-,thick] (y3)
    (v3) edge [-,thick] (y2)
    (v3) edge [-,thick] (y5)
	;

\end{tikzpicture}
\end{center}
\caption{The vertex $v$ of the graph on the left is replaced with a clique $K$ of size $3$ to yield the graph on the right.}
\label{fig:replace}
\end{figure}
We call a graph an \emph{expansion} of a graph $G$ if it is obtained by replacing each vertex in $G$ with a clique (possibly of size one).

Now we are ready to present our main result which completely characterizes niche-realizable graphs with exactly two components.

\begin{Thm}\label{twocomponents}
Let $G$ be a graph with exactly two components $G_1$ and $G_2$.
Then $G$ is niche-realizable if and only if
there are nonnegative integers $a_1$, $b_1$, $a_2$, and $b_2$ such that $2a_1+b_1 \le |V(G_1)|$, $2a_2+b_2 \le |V(G_2)|$, $1 \le a_1+b_1 \le 2^{a_2+b_2-1}$, $1 \le a_2+b_2 \le 2^{a_1+b_1-1}$, and
$G_i$ is an expansion of a complete $(a_i+b_i)$-partite graph with $a_i$ partite sets of size two and $b_i$ partite sets of size one for $i=1, 2$.
\end{Thm}
\begin{proof}
To show the ``only if'' part, suppose that $G$ is niche-realizable.
Then $G$ is the niche graph of a bipartite tournament, say $D$.
Let $(U,V)$ be the bipartition of $D$.
Let $H$ be the subgraph of $G$ induced by a set of representatives of all the equivalence classes under $\equiv_D$. Then, by Lemma~\ref{lem:subgraph}, each component of $H$ is a complete multipartite graph with the partite sets of size at most two.
Then, by definition, $H$ has exactly two components, say $H_1$ and $H_2$.
Moreover, $G_i$ is an expansion of $H_i$ for $i=1,2$.
We denote by $a_i$ the number of partite sets of size two and by $b_i$ the number of partite sets of size one in $H_i$ for $i=1,2$.
Obviously $1 \le a_i+b_i$ and $2a_i+b_i = |V(H_i)| \le |V(G_i)|$ for each $i=1,2$.

Now we take a vertex from each partite set of $H_i$ and denote the set of taken vertices by $X_i$ for each $i=1,2$.
We note that
\begin{itemize}
\item[($\flat$)] any pair of vertices in $X_i$ is adjacent in $H_i$ for each $i=1$, $2$.
\end{itemize}
Let $\AAA = \{ \{Y, X_2 \setminus Y \} \mid Y \subset X_2 \}$.
Then $|\AAA| = 2^{|X_2|-1} = 2^{a_2+b_2-1}$.
Now we define a map $\phi: X_1 \to \AAA$ by $\phi(u) := \{N^+_D(u) \cap X_2, X_2 \setminus N^+_D(u) \}$.
It is obvious that $\phi$ is well-defined.
To show that $\phi$ is one-to-one, suppose that $\phi(u_1)=\phi(u_2)$ for some $u_1$ and $u_2$ in $X_1$. Then $\{N^+_D(u_1) \cap X_2, X_2 \setminus N^+_D(u_1) \}=\{N^+_D(u_2) \cap X_2, X_2 \setminus N^+_D(u_2) \}$, that is, $N^+_D(u_1) \cap X_2=N^+_D(u_2) \cap X_2$ or $N^+_D(u_1) \cap X_2= X_2 \setminus N^+_D(u_2)$.
Before we consider these two cases, we observe the following.

Take a vertex $v \in N^+_D(u_1)$.
Then there exists a vertex $v^*$ in $H_2$ such that $v \equiv_D \! v^*$.
Therefore $v^* \in N^+_D(u_1)$.
Suppose $v^* \not\in X_2$.
Then $v^*$ is contained in a partite set of size two and the other vertex, say $w$, in the partite set belongs to $X_2$.
Then, by Proposition~\ref{prop:notedge}, $v^* \, \RRR_D \, w$ i.e., $N^-_D(v^*)=N^+_D(w)$ and $N^+_D(v^*)=N^-_D(w)$.
Since $v^* \in N^+_D(u_1)$ and $v^* \, \RRR_D \, w$, we have $w \in N^-_D(u_1)$.
Therefore the following are true:
\begin{itemize}
  \item[($\dag$)] If $v^* \notin X_2$, then $N^-_D(v^*)=N^+_D(w)$ and $N^+_D(v^*)=N^-_D(w)$ and $w \in N_D^-(u_1)$ where $w$ is the other vertex in the partite set containing $v^*$.
\end{itemize}

Suppose that $N^+_D(u_1) \cap X_2= X_2 \setminus N^+_D(u_2)$, which is equivalent to $N^+_D(u_1) \cap X_2= N^-_D(u_2) \cap X_2$ or $N^-_D(u_1) \cap X_2= N^+_D(u_2) \cap X_2$.
If $v^* \in X_2$, then $v^* \in N^+_D(u_1) \cap X_2= N^-_D(u_2) \cap X_2$ and so  $v^* \in N^-_D(u_2)$.
Suppose $v^* \notin X_2$.
Then $w \in N^-_D(u_1) \cap X_2$ by ($\dag$).
Since $N^-_D(u_1) \cap X_2 = N^+_D(u_2) \cap X_2$, $w \in N^+_D(u_2)$.
By ($\dag$), $v^* \in N^-_D(u_2)$.
Therefore we have shown that $v^* \in N^-_D(u_2)$  no matter whether $v^* \in X_2$ or $v^* \not\in X_2$.
Since $v \equiv_D \! v^*$, $v \in N^-_D(u_2)$ and so $N^+_D(u_1) \subset N^-_D(u_2)$.
By a symmetric argument, we may show that $N^+_D(u_1) \supset N^-_D(u_2)$, so $N^+_D(u_1) = N^-_D(u_2)$ i.e., $u_1 \, \RRR_D \, u_2$.
Therefore $u_1$ and $u_2$ are not adjacent in $H_1$ by Proposition~\ref{prop:notedge}, which contradicts ($\flat$).
Thus $N^+_D(u_1) \cap X_2=N^+_D(u_2) \cap X_2$, which is equivalent to $N^-_D(u_1) \cap X_2=N^-_D(u_2) \cap X_2$.
If $v^* \in X_2$, then $v^* \in N^+_D(u_1) \cap X_2= N^+_D(u_2)\cap X_2$ and so $v^*\in N^+_D(u_2)$.
Suppose $v^* \not\in X_2$.
Then $w \in N^-_D(u_1) \cap X_2$
by ($\dag$).
Since $N^-_D(u_1) \cap X_2=N^-_D(u_2) \cap X_2$, $w \in N^-_D(u_2)$.
By ($\dag$), $v^* \in N^+_D(u_2)$.
Therefore we have shown that $v^* \in N^+_D(u_2)$ no matter whether $v^* \in X_2$ or $v^* \not\in X_2$.
Since $v \equiv_D \! v^*$, $v \in N^+_D(u_2)$ and so $N^+_D(u_1) \subset N^+_D(u_2)$.
By a symmetric argument, we may show that $N^+_D(u_1) \supset N^+_D(u_2)$, so $N^+_D(u_1) = N^+_D(u_2)$.
Since $u_1$ and $u_2$ are representatives of equivalence classes, $u_1=u_2$.
Thus $\phi$ is one-to-one and so
\[a_1+b_1=|X_1| \le |\AAA| = 2^{a_2+b_2-1}.\]
By symmetry, \[a_2+b_2=|X_2| \le  2^{a_1+b_1-1}.\]
Hence the ``only if'' part is true.

To show the ``if'' part, suppose that there are nonnegative integers $a_1$, $b_1$, $a_2$, and $b_2$ such that $2a_1+b_1 \le |V(G_1)|$, $2a_2+b_2 \le |V(G_2)|$, $1 \le a_1+b_1 \le 2^{a_2+b_2-1}$, $1 \le a_2+b_2 \le 2^{a_1+b_1-1}$, and
$G_i$ is an expansion of a complete $(a_i+b_i)$-partite graph with $a_i$ partite sets of size two and $b_i$ partite sets of size one for each $i=1,2$. We first construct a bipartite tournament $D$ whose niche graph is the disjoint union of a complete $(a_1+b_1)$-partite graph $L_1$ with $a_1$ partite sets of size two and $b_1$ partite sets of size one and a complete $(a_2+b_2)$-partite graph $L_2$ with $a_2$ partite sets of size two and $b_2$ partite sets of size one.

Take a vertex from each partite set of $L_i$,
let $Z_i$ be the set of vertices taken for each $i=1, 2$.
Then $|Z_i|=a_i+b_i$ for each $i=1, 2$.
Without loss of generality, we may assume that $a_1+b_1 \ge a_2+b_2$.
In addition, let $\BBB = \{ \{Y, Z_2 \setminus Y \} \mid Y \subset Z_2 \}$.
Then $|\BBB|=2^{a_2+b_2-1}$.
We note that $|Z_1| \le |\BBB|$ since $a_1 + b_1 \le 2^{a_2+b_2-1}$.
Now let $Q_i$ be $i$-element subset of $Z_2$ for $i=1, \ldots, a_2+b_2$ such that
\[
Q_1 \subset Q_2 \subset Q_3 \subset \cdots \subset Q_{a_2+b_2}=Z_2.
\]
Since $|Z_1| \le |\BBB|$, we may define a one-to-one function $\Psi : Z_1 \to \BBB$ so that $\{\{Q_i,Z_2 \setminus Q_i\} \mid i=1, 2, \ldots, a_2+b_2\} \subset \Psi(Z_1)$.
We construct a bipartite tournament $D$ with bipartition $(V(L_1),V(L_2))$ as follows.
For each vertex $u$ in $Z_1$, we add the arc set
either $[u,Y_u] \cup [Z_2 \setminus Y_u, u]$ or $[Y_u, u] \cup [u, Z_2 \setminus Y_u]$ where $\psi(u)=\{Y_u, Z_2 \setminus Y_u\}$.
If $V(L_1) \setminus Z_1 \neq \emptyset$, then we proceed further to take the following step.
Take a vertex $z$ in $V(L_1) \setminus Z_1$.
Then $z$ belongs to a partite set of size two and the other vertex $z'$ in the partite set is joined to each vertex in $Z_2$ by an arc which has been already added.
Now, for each vertex $z$ in $V(L_1)\setminus Z_1$, we add an arc $(z,w)$ (resp.\ $(w,z)$) if the arc $(w,z')$ (resp.\ $(z',w)$) for $w \in Z_2$ has been added.
If $V(L_2) \setminus Z_2 \neq \emptyset$, then we take a procedure of adding arcs similarly for the case $V(L_1)\setminus Z_1 \neq \emptyset$ to have the bipartite tournament $D$.

For $i=1,2$ and vertices $u$ and $v$ in $Z_i$ belonging to different partite sets, the following is true:
\begin{equation}\label{eqn:psi}
\{N_D^+(u), N_D^-(u)\} \neq \{N_D^+(v), N_D^-(v)\}.
\end{equation}
If $i=1$, \eqref{eqn:psi} is immediately true by the definition of $\Psi$ and the way in which $D$ is constructed.
Suppose $i=2$.
By the definition of $\Psi$, $\{Z_2, \emptyset\} \in \Psi(Z_1)$, so $\emptyset \neq \Psi^{-1}(\{Z_2, \emptyset\}) \subset N_D^+(u) \cap N_D^+(v)$ or $\Psi^{-1}(\{Z_2, \emptyset\}) \subset N_D^-(u) \cap N_D^-(v)$ by the way in which $D$ is constructed.
Thus $N_D^+(u) \neq N_D^-(v)$ and $N_D^-(u) \neq N_D^+(v)$.
By the definition of $Q_j$, $|Q_{j^*} \cap \{u, v\}|=1$ for some $j^* \in \{1, \ldots, a_2+b_2\}$.
By the definition of $\Psi$, $\Psi^{-1}(\{Q_{j^*}, Z_2 \setminus Q_{j^*}\}) \neq \emptyset$.
We may assume that $u \in Q_{j^*}$.
Then $v \in Z_2 \setminus Q_{j^*}$.
Thus, by the definition of $\Psi$ and the way in which $D$ is constructed, exactly one of the following is true:
\[\Psi^{-1}(\{Q_{j^*}, Z_2 \setminus Q_{j^*}\}) \subset N_D^+(u) \text{ and }  \Psi^{-1}(\{Q_{j^*}, Z_2 \setminus Q_{j^*}\}) \not\subset N_D^+(v);\]
\[ \Psi^{-1}(\{Q_{j^*}, Z_2 \setminus Q_{j^*}\}) \subset N_D^+(v) \text{ and }  \Psi^{-1}(\{Q_{j^*}, Z_2 \setminus Q_{j^*}\}) \not\subset N_D^+(u),\]
which implies $N_D^+(u) \neq N_D^+(v)$ and $N_D^-(u) \neq N_D^-(v)$.
Therefore \eqref{eqn:psi} is true for $i=2$.

We will show that the disjoint union $L_1 \cup L_2$ of $L_1$ and $L_2$ is the niche graph of $D$. Let $G'$ be the niche graph of $D$.
By the way of construction, for $z$ and $z'$ in the same partite set of size $2$ of $L_1$ (resp.\ $L_2$), $z \, \RRR_D \, z'$ so the two vertices in the same partite set of $L_1$ (resp.\ $L_2$) form an independent set by Proposition~\ref{prop:notedge} in $G'$.
Take two vertices $u_1$ and $u_2$ from different partite sets of $L_i$ for $i=1,2$.
By \eqref{eqn:psi}, $\{N_D^+(u_1'), N_D^-(u_1')\} \neq \{N_D^+(u_2'), N_D^-(u_2')\}$  where $u_j'$ is the vertex in the partite set containing $u_j$ that belongs to $Z_i$ for $j=1,2$.
By construction of $D$, $\{N_D^+(u_i), N_D^-(u_i)\}=\{N_D^+(u_i'), N_D^-(u_i')\}$ for $i=1,2$, so $\{N_D^+(u_1), N_D^-(u_1)\} \neq \{N_D^+(u_2), N_D^-(u_2)\}$.
Thus $u_1$ and $u_2$ are adjacent in $G'$.
Consequently we have shown that $G'=L_1 \cup L_2$.

By the definition of expansion, $G$ is obtained by replacing each vertex $v$ of $L_1 \cup L_2$ with a clique $K_v$. We define a digraph $D^*$ as follows:
\[V(D^*)=\bigcup_{v \in V(D)}K_v;\]
\[A(D^*)=\bigcup_{(u,v) \in A(D)}[K_u,K_v].\]
It is obvious that $D^*$ is a bipartite tournament. By the definition of expansion, it is easy to see that $G$ is the niche graph of $D^*$.
\end{proof}

From Theorem~\ref{twocomponents}, we may derive basic niche-realizable pairs.
To do so, we introduce the following notations.
For a complete multipartite graph $H$ with partite sets of size at most two, we denote by $a(H)$ and $b(H)$ the number of partite sets of size two and the number of partite sets of size one, respectively.
For an expansion $G$ of a complete multipartite graph with partite sets of size at most two, we let
\begin{align*}
X(G) = \{(a(H),b(H)) \mid &\text{ $G$ is an expansion of a complete multipartite graph $H$} \\
&\text{ \ with partite sets of size at most $2$} \}.
\end{align*}

\begin{Cor}\label{prop:KmKn}
The pair $(K_m, K_n)$ is niche-realizable for positive integers $m$ and $n$.
\end{Cor}
\begin{proof}
Complete graphs are expansions of $K_1$.
Therefore $(0,1) \in X(K_l)$ for any $l \ge 1$.
For $(G_1,G_2)=(K_m,K_n)$ and $a_1=a_2=0$ and $b_1=b_2=1$, the inequalities $2a_1+b_1 \le |V(G_1)|$, $2a_2+b_2 \le |V(G_2)|$, $1 \le a_1+b_1 \le 2^{a_2+b_2-1}$, and $1 \le a_2+b_2 \le 2^{a_1+b_1-1}$ are satisfied.
Thus, by Theorem~\ref{twocomponents}, the statement is true.
\end{proof}

\begin{Cor}\label{cor:path}
Let $m$ and $n$ be positive integers.
Then the pair $(P_m, P_n)$ is niche-realizable if and only if $(m,n) \in \{(1,1), (1,2), (2,1), (2,2), (2,3), (3,2), (3,3)\}$.
\end{Cor}
\begin{proof}
By Proposition~\ref{prop:induced P3}, $(P_m, P_n)$ is not niche-realizable if $m \ge 4$ or $n \ge 4$.
It is easy to check that the path graphs $P_1$, $P_2$, and $P_3$ are expansions of complete multipartite graphs with
$X(P_1) = \{(0,1)\}$, $X(P_2) = \{(0,1), (0,2)\}$, and $X(P_3) = \{(1,1)\}$.
For $(G_1,G_2)=(P_m,P_n)$, it is easy to check that there exist $(a_1,b_1) \in X(P_m)$, $(a_2,b_2) \in X(P_n)$ satisfying the inequalities
$2a_1+b_1 \le |V(G_1)|$, $2a_2+b_2 \le |V(G_2)|$, $1 \le a_1+b_1 \le 2^{a_2+b_2-1}$, and $1 \le a_2+b_2 \le 2^{a_1+b_1-1}$ if and only if  $(m,n) \in \{(1,1), (1,2), (2,1), (2,2), (2,3), (3,2), (3,3)\}$.
Thus, by Theorem~\ref{twocomponents}, the statement is true.
\end{proof}

\begin{Cor}\label{cor:cycle}
Let $m$ and $n$ be integers with $m,n \ge 3$.
Then the pair $(C_m, C_n)$ is niche-realizable if and only if $(m,n) \in \{(3,3), (3,4), (4,3), (4,4)\}$.
\end{Cor}
\begin{proof}
By Proposition~\ref{prop:induced P3}, $(C_m, C_n)$ is not niche-realizable if $m \ge 5$ or $n \ge 5$.
It is easy to check that the cycles $C_3$ and $C_4$ are expansions of complete multipartite graphs and that $X(C_3) = \{(0,1), (0,2), (0,3)\}$ and $X(C_4) = \{(2,0)\}$.
For $(G_1,G_2)=(C_m,C_n)$, it is easy to check that there exist $(a_1,b_1) \in X(C_m)$, $(a_2,b_2) \in X(C_n)$ satisfying the inequalities
$2a_1+b_1 \le |V(G_1)|$, $2a_2+b_2 \le |V(G_2)|$, $1 \le a_1+b_1 \le 2^{a_2+b_2-1}$, and $1 \le a_2+b_2 \le 2^{a_1+b_1-1}$ if and only if  $(m,n) \in \{(3,3), (3,4), (4,3), (4,4)\}$.
Thus, by Theorem~\ref{twocomponents}, the statement is true.
\end{proof}

We call the graph resulting from identifying the vertices in each critical clique in $G$  the \emph{condensation} of $G$ and denote it by $G^*$. We note that $G$ is an expansion of $G^*$.

We derive a necessary condition for a graph being niche-realizable in terms of the condensation of a graph, which might be more intuitive.

\begin{Lem}\label{lem:condensation11}
Let $G$ be a graph and $G'$ be a graph obtained by identifying some homogeneous vertices of $G$.
Then $(G')^* = G^*$.
\end{Lem}
\begin{proof}
It is obvious by definition that, if $u^*$ is the vertex in $G'$ which is obtained by identifying $u$ and some vertices homogeneous with $u$ other than $v$, then $u$ and $v$ are homogeneous in $G$ if and only if $u^*$ and $v$ are homogeneous in $G'$.
The lemma immediately follows from this observation.
\end{proof}

\begin{Thm}\label{thm:componentstucture:new}
Each component of the condensation of a
niche-realizable graph is a complete multipartite graph with the partite sets of size at most two among which there exists at most one partite set of size one.
\end{Thm}

\begin{proof}
Let $G$ be a niche-realizable graph.
Then $G$ is the niche graph of a bipartite tournament $D$.
Let $(U,V)$ be the bipartition of $D$, and $Z$ be a set of representatives of all the equivalence classes under $\equiv_D$ given in \eqref{eqn:equivalence}. Let $G/\!\equiv_D$ be the graph obtained from $G$ by identifying the vertices in each of equivalence classes under $\equiv_D$.
Then $G/\!\equiv_D$ is isomorphic to the subgraph $H$ of $G$ induced by $Z$. 
Thus, by Lemma~\ref{lem:subgraph}, each component of $G/\!\equiv_D$ is a complete multipartite graph with the partite sets of size at most two.
By Lemma~\ref{lem:condensation11} and ($\star$), $\left(G/\!\equiv_D\right)^*=G^*$.
It is easy to check that the condensation of a complete multipartite graph with the partite sets of size at most two is a complete multipartite graph with the partite sets of size at most two.
Therefore each partite set of $G^*$ has at most two vertices.
Furthermore, since the vertices from a partite set of size one are homogeneous, there exists at most one partite set of size one in $G^*$.
\end{proof}

\section{Noteworthy properties of niche-realizable graphs}
In this section, we apply Theorem~\ref{twocomponents} to find meaningful properties of niche-realizable graphs.

\begin{Prop}\label{prop:hole}
Let $G$ be a niche-realizable graph. Then neither $G$ nor $\overline{G}$ has a hole of length greater than or equal to five.
\end{Prop}
\begin{proof}
Suppose that $G$ contains a hole of length $l$.
Then $G^*$ contains a hole $C$ of length $l$.
By Theorem~\ref{thm:componentstucture:new}, $G^*$ is a complete multipartite graph.
Since $C$ is an induced subgraph of $G^*$, $C$ is a complete multipartite and so $l \le 4$.
Therefore $G$ does not contain a hole of length greater than or equal to five.

Suppose to the contrary that $\overline{G}$ has a hole $C$ of length $n$ for some integer $n \ge 5$.
By Proposition~\ref{prop:complement forbidden path}, $n=5$.
Then the complement of $C$ is a hole of length five, so $G$ has a hole of length five, which is a contradiction.
\end{proof}

The strong perfect graph theorem states that a graph $G$ is perfect if and only if neither $G$ nor $\overline{G}$ has a hole of odd length.
Therefore, by the above proposition, we have the following result.

\begin{Cor}
A niche-realizable graph is perfect.
\end{Cor}

\begin{Prop}
Let $G$ be a niche-realizable graph.
Then $G$ is chordal if and only if one of the following is true:
\begin{itemize}
  \item it has exactly three or exactly four components each of which is complete;
  \item it has exactly two components each of which is a complete graph or an expansion of a path graph of length two.
\end{itemize}
\end{Prop}
\begin{proof}
By Theorem~\ref{thm:component:new}, it is sufficient to assume that $G$ has exactly two components.
Let $G_1$ and $G_2$ be the components of $G$.
Since $G$ is niche-realizable, by Theorem~\ref{twocomponents}, $G_i$ is an expansion of complete multipartite graph $H_i$ with partite sets of size one or two for each $i \in \{1,2\}$.
By Proposition~\ref{prop:hole}, $G$ is chordal if and only if $G_i$ does not contain a hole of length four for each $i \in \{1,2\}$ if and only if $H_i$ does not contain a hole of length four for each $i \in \{1,2\}$.
Fix $i \in \{1,2\}$.
It is easy to see that $H_i$ does not contain a hole of length four if and only if $H_i$ has at most one partite set of size two.
If $H_i$ does not have a partite set of size two, then $H_i$ is a complete graph and so $G_i$ is a complete graph.
Suppose that $H_i$ contains exactly one partite set of size two.
Since $H_i$ is connected, there is another partite set of size one and so $H_i$ is an expansion of an induced path of length two, which implies that so is $G_i$.
Thus we may conclude that $H_i$ has at most one partite set of size two for each $i \in \{1, 2\}$ if and only if each component of $G$ is a complete graph or an expansion of a path graph of length two.
\end{proof}

We denote the size of a maximum clique in a graph $G$ by $\omega(G)$.

\begin{Prop}\label{prop:omega}
For a niche-realizable graph $G$, each component of $G$ contains at most $2\omega(G)$ vertices and $|V(G)| \le 4\omega(G)$.
\end{Prop}
\begin{proof}
Suppose that $G$ has three or four components.
Then, by Theorem~\ref{thm:component:new}, each component is complete, so each component has at most $\omega(G)$ vertices.
Thus $|V(G)| \le 4\omega(G)$.
Suppose that $G$ has exactly two components $G_1$ and $G_2$.
To the contrary, we suppose that one of the two components has at least $2\omega(G)+1$ vertices.
Without loss of generality, we may assume that $G_1$ has at least  $2\omega(G)+1$ vertices.
Then, by Theorem~\ref{twocomponents}, $G_1$ is an expansion of a complete $l$-partite graph $H$ with partite sets of size at most two for a positive integer $l$.
When $G_1$ was obtained from $H$, for each $i$, $1 \le i \le l$, let $s_i$ and $t_i$ be nonnegative integers such that $s_i \ge t_i$ and the vertices in the $i$th partite set of $H$ have been replaced with a clique $S_i$ of size $s_i$ and, as long as $t_i > 0$, with a clique $T_i$ of size $t_i$ ($t_i=0$ if the $i$th partite set has one vertex).
By our assumption, $\sum_{i=1}^{l} (s_i+t_i) \ge 2\omega(G)+1$ and $\sum_{i=1}^{l}s_i \ge \sum_{i=1}^{l}t_i$.
Therefore
\[
\sum_{i=1}^{l}s_i \ge \omega(G)+1.
\]
For $i \neq j$, the vertices in $H$ corresponding to $S_i$ and $S_j$ belong to distinct partite set and so are adjacent in $H$.
Thus $\bigcup_{i=1}^{l}V(S_i)$ forms a clique in $G$ of size at least $\omega(G)+1$.
Hence we have reached a contradiction and the proposition statement is true.
\end{proof}

\noindent
Since $\omega(G) \le 4$ for a planar graph $G$, the following corollary is immediately true by Proposition~\ref{prop:omega}.

\begin{Cor}\label{planar}
Let $G$ be a planar niche-realizable graph.
Then each component of $G$ contains at most $8$ vertices and $|V(G)| \le 16$.
\end{Cor}

\begin{Prop}
For a niche-realizable graph $G$, the size of a maximum matching of $G$ is greater than or equal to ${|V(G)-4|} \over {2}$.
\end{Prop}
\begin{proof}
If $G$ has three or four components, then there are at most four unsaturated vertices for a maximum matching of $G$ by Theorem~\ref{thm:component:new}.
Suppose $G$ has exactly two components $G_1$ and $G_2$.
Fix $i \in \{1, 2\}$.
Suppose, to the contrary, that there exists a maximum matching $M$ such that at least three vertices in $G_i$ are $M$-unsaturated.
Then, since $\alpha(G_i) \le 2$ by Theorem~\ref{thm:alpha}, there exist two adjacent vertices among the $M$-unsaturated vertices in $G_i$.
Adding the edge joining the two vertices to $M$ creates a matching of size $|M|+1$, which contradicts the maximality of $M$.
Therefore we may conclude that there are at most two unsaturated vertices in $G_i$ for each maximum matching of $G$ for each $i=1$, $2$, and so the proposition statement is valid.
\end{proof}

\begin{Prop}
Given a niche-realizable graph $G$, if $G[S]$ is a connected non-complete regular graph for a vertex set $S$, then $G[S]$ is an expansion of a complete multipartite graph on $l$ partite sets of size two for a positive integer $l$ which is obtained by replacing each vertex with a clique of size $t$ for a positive integer $t$.
\end{Prop}
\begin{proof}
Suppose that $G[S]$ is a connected non-complete regular graph for a vertex set $S$.
By Corollaries~\ref{cor:at least two components} and~\ref{cor:components4} and Theorem~\ref{thm:component:new}, $G$ has exactly two components.
Thus, by Theorem~\ref{twocomponents}, each component of $G$ is an expansion of a complete multipartite graph with partite sets of size at most two.
Since $G[S]$ is connected, it is an induced subgraph of one of the components.
Since an induced subgraph of a complete multipartite graph is also a complete multipartite graph, $G[S]$ is an expansion of a complete $l$-partite graph $H$ with partite sets of size at most two for a positive integer $l$.
When $G[S]$ was obtained from $H$, for each $i$, $1 \le i \le l$, let $s_i$ and $t_i$ be nonnegative integers such that $s_i \ge t_i$ and the vertices in the $i$th partite set of $H$ have been replaced with cliques of sizes $s_i$ and $t_i$ ($t_i=0$ if the $i$th partite set has one vertex).
Then the degree of a vertex in $K_{s_{i}}$ which replaces a vertex in the $i$th partite set in $G[S]$ is $|S|-t_i-1$ for each $i$, $1 \le i \le l$.
Since $G[S]$ is regular, $t_i = t$ for some nonnegative integer $t$ for each $i$, $1 \le i \le l$.
If $t=0$, then $G[S]$ is complete, which contradicts the assumption that $G[S]$ is non-complete.
Therefore $t>0$.
Then the degree of a vertex in $K_{t_i}$ which replaces a vertex in the $i$th partite set in $G[S]$ is $|S|-s_i-1$.
Since $G[S]$ is regular, $|S|-s_i-1 = |S|-t_i-1$ and so $s_i=t_i=t$ for each $i$, $1 \le i \le l$.
Thus the proposition follows.
\end{proof}

Let $G$ be a connected graph.
A \emph{vertex cut} of $G$ is a subset $V'$ of $V(G)$ such that $G - V'$ is disconnected.
A \emph{$k$-vertex cut} is a vertex cut of $k$ elements.
If $G$ is not complete, then the \emph{connectivity} $\kappa(G)$ of $G$ is the minimum $k$ for which $G$ has a $k$-vertex cut.
Otherwise, $\kappa(G)$ is defined to be $|V(G)| - 1$.

\begin{Thm}\label{thm:hamilton}
Every component of a niche-realizable graph has a Hamilton path and every $2$-connected component of a niche-realizable graph is hamiltonian.
\end{Thm}
\begin{proof}
Let $G$ be a niche-realizable graph and $X$ be a component of $G$.
If $X$ is complete, then $X$ is obviously hamiltonian.
Suppose that $X$ is not complete.
Then, by Corollaries~\ref{cor:at least two components} and~\ref{cor:components4} and Theorem~\ref{thm:component:new}, $G$ has exactly two components.
In addition, by Theorem~\ref{twocomponents},
$X$ is an expansion of a complete $l$-partite graph $H$ with partite sets of size at most two.
Furthermore, by Theorem~\ref{thm:alpha}, $\alpha(X) \le 2$.
If $X$ is $2$-connected, then $X$ has at least three vertices and $\kappa(X) \ge 2$, and so $\kappa(X) \ge \alpha(X)$, which guarantees that $X$ is hamiltonian by the Chv\'{a}tal-Erd\H{o}s Theorem.

Now we consider the case where $\kappa(X) = 1$.
Then $X$ has a cut vertex $w$ as $X$ is not $K_2$.
Now there exist two nonadjacent vertices $u$ and $v$ in $X$ such that $u$ and $v$ have exactly one common neighbor $w$.
For notational convenience, we denote by $a^*$ the vertex in $H$ which was replaced with the clique containing a vertex $a$ in $X$.
We note that $u^*$ and $v^*$ belong to the same partite set, say $Q$, in $H$.
We add a vertex $x$ to $X$ and edges incident to $x$ to obtain a new graph $X'$ so that $x$ and $w$ are homogeneous in $X'$.
Since $w$ is replaced with the clique $\{x, w\}$, we may regard $X'$ as an expansion of $X$ and so $X'$ is still an expansion of $H$.
Therefore, by Theorem~\ref{twocomponents}, $X'$ is niche-realizable.
Take two nonadjacent vertices $y$ and $z$ in $X'$.
Then $y^*$ and $z^*$ belong to the same partite set in $H$.
If $y^*$ and $z^*$ in $Q$, then $ywz$ and $yxz$ are distinct internally disjoint $(y, z)$-paths in $X'$.
If $y^*$ and $z^*$ belong to a partite set in $H$ distinct from $Q$, then $yuz$ and $yvz$ are distinct internally disjoint $(y, z)$-paths in $X'$.
Therefore $X'$ is $2$-connected and so, by the argument above, hamiltonian, which immediately implies that $X=X'-x$ has a  Hamilton path.
\end{proof}

\section*{Acknowledgments}

The first and third author's research was supported by
the National Research Foundation of Korea (NRF) funded by the Korea government (MEST) (No.\ NRF-2017R1E1A1A03070489) and by the Korea government(MSIP)
(No.\ 2016R1A5A1008055).
The first and second author's research was supported by Basic Science Research Program through the National Research Foundation of Korea (NRF) funded by the Ministry of Education (NRF-2018R1D1A1B07049150).

\end{document}